\documentclass[12pt]{article}
\usepackage{graphicx}
\usepackage{latexsym,amssymb}
\usepackage{amsthm}
\usepackage{indentfirst}
\usepackage{amsmath}
\usepackage{color}
\usepackage{xcolor}
\usepackage{tikz}
\usepackage{stmaryrd}
\usepackage{epsfig}
\usepackage{amsmath}
\usepackage{amssymb}
\usepackage{amscd}
\usepackage{graphicx}
\usepackage{cite}
\usepackage{mathtools}      
\usepackage{mathabx}
\usepackage{pstricks}

\usepackage{tocvsec2}

\usepackage{hyperref}

\input xy
\xyoption{all}

 \def\NN{{\mathbb N}} 

 \def\RR{{\mathbb R}} 
\def\TT{{\mathbb T}}

   \def\cM{{\cal M}} 
    
    \def\cU{{\cal U}}
    
\def\cE{{\cal E}}    
\def\cF{{\cal F}}  \def\cL{{\cal L}}

\newtheorem{theoremalph}{Theorem}

\newtheorem{thm}{Theorem}[section]
\newtheorem{lem}[thm]{Lemma}

\newtheorem{pro}[thm]{Proposition}

\newtheorem{conj}[thm]{Conjecture}

\newtheorem{rem}[thm]{Remark}
\newtheorem*{Claim}{Claim}
\newtheorem*{Acknowledgements}{Acknowledgements}

\numberwithin{equation}{section}

\begin{document}

\title{Lyapunov optimizing measures and periodic measures for $C^2$ expanding maps}

\author{Wen Huang\footnote{W.H. was partially supported by NSFC (12090012, 12031019, 11731003)} \and Leiye Xu\footnote{L.X. was partially supported by NSFC (12031019, 11801538,
11871188)} \and Dawei Yang\footnote{D.Y.  was partially supported by NSFC (11822109, 11790274).}}

\maketitle


\begin{abstract}
We prove that there exists an open and dense subset $\mathcal{U}$ in the space of $C^{2}$ expanding self-maps of the circle $\mathbb{T}$ such that the Lyapunov minimizing measures of any $T\in\cU$ are uniquely supported on a periodic orbit.

This answers a conjecture of Jenkinson-Morris in the $C^2$ topology.
\end{abstract}
\maketitle
	\section{Introduction}
	\subsection{Main theorems}
The ergodic optimization problem has connections with Lagrangian Mechanics, Thermodynamical Formalism, Multifractal Ananlysis, and Control Theory (see \cite{J2019}).
In the generic chaotic setting, it has been conjectured by Yuan and Hunt \cite{YuH99} that for an Axiom A or uniformly expanding system $T$ and a (topologically) generic smooth function $f$, there exists an optimal periodic orbit.
Contreras \cite{C} has made substantial contributions to Yuan-Hunt's conjecture. Later on,  the papers \cite{HLMXZ,HLMXZ2,MSV19} progressed a lot in this direction.

In Yuan-Hunt's conjecture, the function $f$ is not strongly related to the system $T$.
The aim of this paper is to consider the optimal measures of some quantities very related to the dynamical system. One of the most interesting quantities may be the \emph{Lyapunov exponent}. The measures optimize Lyapunov exponents are said to be \emph{Lyapunov optimal measures}. This notion was given by Contreras-Lopes-Thieullen \cite{CLT}. We will show that the Lyapunov minimizing/maximizing measures of generic 1-dimensional expanding self-maps are supported on periodic orbits for the $C^2$ topology.

Let $\TT=\mathbb{R}/\mathbb{Z}$ be the circle and $T: \mathbb{T}\rightarrow \mathbb{T}$ be a $C^1$ self-map. Let $\mathcal{M}_{inv}(T)$  (resp. $\mathcal{M}_{erg}(T)$) be the set of $T$-invariant (resp. $T$-ergodic) Borel probability measures. For any $\mu\in  \mathcal{M}_{inv}(T)$, define its \emph{Lyapunov
exponent} as
$$\lambda_T(\mu) :=\int \log |DT| {d}\mu.$$
Lyapunov exponents are vary important dynamical quantities. We are interested in seeking which measures minimize or maximize the Lyapunov exponents. Define
$$\alpha(T):=\inf_{\nu\in \cM_{inv}(T)} \lambda_T(\nu),~~~\beta(T):=\sup_{\nu\in \cM_{inv}(T)} \lambda_T(\nu).$$
An invariant measure $\mu$ is said to be a \emph{Lyapunov minimizing measure} if $\alpha(T)=\lambda_T(\mu)$; it is said to be a \emph{Lyapunov maximizing measure} if $\beta(T)=\lambda_T(\mu)$.

\medskip

Lyapunov minimizing/maximizing measures may be very difficult to describe for any self-map. However, it was imagined for these measures are periodic for generic expanding self-maps. A self-map $T:\TT\to\TT$ is \emph{expanding} if there are $C>0$ and $\lambda>1$ such that $\|DT^n(x)\|\ge C\lambda^n$ for any $x\in\TT$.

For any two self-maps $S$ and $T$, the $C^k$-distance between $S$ and $T$ is defined to be
$$d_{C^k}(S,T)=\sum_{i=0}^k d_{C^0}(D^iS,D^iT).$$
Given $\chi\in(0,1]$, the $C^{k,\chi}$-distance between $S$ and $T$ is defined to be
$$d_{C^{k,\chi}}(S,T)=d_{C^k}(S,T)+\sup_{x\neq y}\frac{d_{C^0}(D^kS,D^kT)}{|x-y|^{\chi}}.$$

Let $\cE^{k}(\TT)$ ($\cE^{k,\chi}(\TT)$) be the space of $C^k$ ($C^{k,\chi}$) expanding self-maps endowed with the $C^k$-distance ($C^{k,\chi}$-distance).


\begin{theoremalph}\label{Thm:main-lyapunov}
There is a dense open set ${\mathcal U}\subset\cE^{2}(\TT)$ such that the Lyapunov minimizing measure of $T\in\cU$ is unique and supported on a periodic orbit.
\end{theoremalph}

The proof of Theorem~\ref{Thm:main-lyapunov} is mainly based on a Lipschitz-$C^1$ version.

\begin{theoremalph}\label{Thm:main-lyapunov-Holder}
 There is a dense open set ${\mathcal U}\subset\cE^{1,1}(\TT)$ such that the Lyapunov minimizing measure of $T\in\cU$ is unique and supported on a periodic orbit.
\end{theoremalph}

One can also get a H\"older-$C^1$ version of Theorem~\ref{Thm:main-lyapunov-Holder}.

\begin{theoremalph}\label{Thm:main-lyapunov-real-Holder}
Assume that $\chi\in(0,1]$. There is a dense open set ${\mathcal U}\subset\cE^{1,\chi}(\TT)$ such that the Lyapunov minimizing measure of $T\in\cU$ is unique and supported on a periodic orbit.
\end{theoremalph}
Since the proof of Theorem~\ref{Thm:main-lyapunov-real-Holder} follows almost the same line of the proof of Theorem~\ref{Thm:main-lyapunov-Holder}, it is omitted.

\medskip

Theorem~\ref{Thm:main-lyapunov} answers a conjecture of Jenkinson-Morris \cite[Conjecture 1]{JM08} positively in the $C^2$ topology.
\begin{conj} \cite{JM08}\label{con1} For integer $k\ge 2$, a generic  $T\in \mathcal{E}^k$ has a unique Lyapunov minimizing
measure, and this measure is supported on a periodic orbit of $T$.
\end{conj}
Note that  the conjecture of Jenkinson-Morris when $k>2$ is still open.

\medskip

Contreras-Lopez-Thieullen \cite{CLT} has proved that  for any $T$  in some dense open subset of $\bigcup_{1>\beta>\alpha}\cE^{1+\beta}$ endowed with the $C^{1,\alpha}$-distance, the Lyapunov minimizing/maximizing measures of $T$ are unique and periodic.

In $C^1$ topology, the situation is completely different. It has been proved by Jenkinson-Morris \cite{JM08} that for generic $C^1$ expanding self-maps on $\TT$, the Lyapunov minimizing is unique but has full support.

\subsection{Discussions in the higher-dimensional case}

In the higher-dimensional case, one interesting problem is to consider the ergodic optimization problem of the upper Lyapunov exponents. Let $M$ be a $d$-dimensional Riemannian manifold without boundary. Let $T:~M\to M$ be a $C^1$ self-map. Given an ergodic measure $\mu$ of $T$, as in \cite[Section C.1]{BDV05}, there is a measurable filtration for $\mu$-almost every point $x\in M$,
$$T_xM=E_1(x)\supset E_2(x)\supset \cdots \supset E_k(x)\supset E_{k+1}=\{0\}$$
and constants $\lambda_1>\lambda_2>\cdots>\lambda_k$ such that for any $1\le i\le k$ and for any $v\in E_i\setminus E_{i-1}$, one has that
$$\lim_{n\to\infty}\frac{1}{n}\log\|DT^n(v)\|=\lambda_i.$$
$\lambda_1$ is said to be the \emph{upper Lyapunov exponent} of $\mu$. An invariant measure $\mu$ is said to be the maximizing measure of $\lambda_1$ if
$\lambda_1(\mu)=\sup_{\nu\in\cM_{inv}(T)} \lambda_1(\nu)$. An invariant measure $\mu$ is said to be the minimizing measure of $\lambda_i$ if
$\lambda_1(\mu)=\inf_{\nu\in\cM_{inv}(T)} \lambda_1(\nu)$. By Cao \cite{Cao}, the maximizing measure of $\lambda_1$ does exist. Symmetrically, one knows the existence of minimizing measure of $\lambda_k$ (the lower Lyapunov exponent). We  have the following conjectures:

\begin{conj}Let $M$ be a $d$-dimensional compact Riemannian manifold. For an integer $k\ge 2$, for a $C^k$ generic  expanding self-map $T$, the upper Lyapunov exponent $\lambda_1$ has a unique Lyapunov maximizing
measure, and this measure is supported on a periodic orbit of $T$.

For $k=1$, for a $C^1$ generic expanding self-map $T$, the upper Lyapunov exponent admits a unique maximizing measure, which has zero entropy and full support.
\end{conj}

Although we do not know the existence of minimizing measure of the upper Lyapunov exponent $\lambda_1$, one can still formulate the following conjecture.
\begin{conj}
For generic expanding self-map $T$ on a manifold $M$, the minimizing measure of the upper Lyapunov exponent $\lambda_1$ exists, is unique and has zero entropy.

\end{conj}

One can still have results in the higher-dimensional case on the sum of Lyapunov exponents. Given an ergodic measure $\mu$, denote by
$$\lambda_{sum}=\sum_{i=1}^k\lambda_i(\dim E_i-\dim E_{i+1}).$$
By \cite[Proposition 1.3, Theorem 1.6]{Rue79}, to find the optimal measures of $\lambda_{sum}$ is equivalent to find the optimal measures of the continuous function $\log|{\rm Det}(T)|$. Thus, it is essentially the same as the one-dimensional case. One has the following theorems.

Denote by $\cE^{k}(M)$ and $\cE^{k,\chi}(M)$ the spaces of $C^k$ expanding maps and of $C^{k+\chi}$ expanding maps, respectively.
\begin{theoremalph}\label{Thm:main-lyapunov-high}
There is a dense open set ${\mathcal U}\subset\cE^{2}(M)$ such that the minimizing measure with respect to $\lambda_{sum}$ of $T\in\cU$ is unique and supported on a periodic orbit.
\end{theoremalph}

\begin{theoremalph}\label{Thm:main-lyapunov-real-Holder-high}
Assume that $\chi\in(0,1]$. There is a dense open set ${\mathcal U}\subset\cE^{1,\chi}(M)$ such that the minimizing measure with respect to $\lambda_{sum}$ of $T\in\cU$ is unique and supported on a periodic orbit.
\end{theoremalph}

The proof of these two theorems follow almost the same lines of the above ones, hence are omitted.

In the $C^1$ topology, the optimization problem of $\lambda_{sum}$ has been considered in \cite{MT}.

\begin{Acknowledgements}

We are grateful to Jianyu Chen for providing some references.

\end{Acknowledgements}
\section{Proof of Theorem~\ref{Thm:main-lyapunov}}
In this section, we are going to prove Theorem~\ref{Thm:main-lyapunov}. Denote by $\cL^{-}(T)$/$\cL^+(T)$ the set of Lyapunov minimizing/maximizing measures of $T$, respectively.

We need the following version of Ma\~n\'{e}'s Lemma.
	\begin{lem}\label{lem-Mane} Let  $T$ be a $C^{1,1}$ expanding self-map of $\mathbb{T}$.  Then there exists a Lipschitz map $f$ from $\TT$ to $\RR$ such that
	\begin{align}\label{AA}\bigcup_{\nu\in \cL^-(T)}{\rm supp}(\nu)\subset \{y\in\mathbb{T}:F(y)=\inf_{x\in\mathbb{T}}F(x)=\alpha(T)\}.\end{align}
	where $F(x)=f(T(x))-f(x)+\log\|DT(x)\|$.
	\end{lem}
\begin{proof}
	Notice that $\log\|DT\|$ is Lipschitz since $T$ is $C^{1,1}$.
This Lemma follows immediately from  the classical Ma\~n\'e's Lemma for the expanding self-map $T$ and the Lipschitz function $\log\|DT\|$ (see
\cite{Bou1,S,Bou,CLT,CG} for various versions and approaches).
	\end{proof}
	
Expanding self-maps on the circle are structurally stable: this was proved by Shub \cite{Shu69}. We can also have the information on the conjugacy maps, see \cite[Lemma 2]{JM08} and \cite[Proposition 5.1.6]{PrU10} for a precise proof.
\begin{thm}\label{Thm:conjugacy}
Let $S_0$ be a $C^1$ expanding self-map of $\TT$. For any $\widetilde\varepsilon_0>0$, there is $\widetilde\varepsilon>0$ such that for any $S$, if $d_{C^1}(S,S_0)<\widetilde\varepsilon$, then there is a homeomorphism $\pi_S:~\TT\to\TT$ such that
\begin{itemize}
\item $d_{C^0}(\pi_S,Id)<\widetilde\varepsilon_0$,
\item $\pi_S\circ S_0=S\circ \pi_S$.

\end{itemize}

\end{thm}
%

Let $T$ be a $C^1$ expanding self-map of $\mathbb{T}$ with $\|DT(x)\|>1$ for all $x\in\mathbb{T}$. Let $\Gamma$ be a periodic orbit of $T$. Define the gap of $\Gamma$ by
$$G(\Gamma)=\left\{\begin{array}{ll}
\frac{1}{20\max_{x\in\TT}\|DT(x)\|}, & \text{ if } \#\Gamma=1,\\
\min_{x,y\in\Gamma, x\not=y}d(x,y), & \text{ others }.
\end{array}
\right.$$

One has the following expansive-like Lemma.
\begin{lem}\label{shadow} Let  $T$ be a $C^1$ expanding self-map of $\mathbb{T}$ with $\|DT(x)\|>1$  for  each $x\in\mathbb{T}$ and $\Gamma$ be a periodic orbit of $T$. If $z\in\mathbb{T}$ satisfies
	$$d(T^iz,\Gamma)<\frac{G(\Gamma)}{2\max_{x\in\TT}\|DT(x)\|}\text{ for all }i\in\mathbb{N}\cup\{0\},$$
	then $z\in\Gamma$.
\end{lem}
\begin{proof}
We will prove by contradiction and assume that $z\notin\Gamma$. One can find $p\in\Gamma$ such that $d(z,p)=d(z,\Gamma)$. {If $\#\Gamma=1$, by the expansion property, one can find $m\in\NN$ such that
$$d(T^m(z),T^m(p))\ge\frac{1}{2\max_{x\in\TT}\|DT(x)\|}>\frac{G(\Gamma)}{2\max_{x\in\TT}\|DT(x)\|}.$$}
Now one can assume that $\#\Gamma>1$. By the expansion property, there is $m\in\NN$ such that
$$\frac{G(\Gamma)}{2\max_{x\in\TT}\|DT(x)\|}\le d(T^m(z),T^m(p))< \frac{G(\Gamma)}{2}.$$
For any $q\in\Gamma\setminus\{p\}$, one has that
$$d(T^m(z),T^m(q))\ge d(T^m(p),T^m(q))-d(T^m(z),T^m(p))\ge G(\Gamma)-G(\Gamma)/2\ge G(\Gamma)/2.$$
Thus $$d(T^m(z),\Gamma)=d(T^m(z),T^m(p))\ge \frac{G(\Gamma)}{2\max_{x\in\TT}\|DT(x)\|}.$$
This contradicts to the assumption.
\end{proof}

\subsection{Periodic orbits with large inner distance}

To prove Theorem~\ref{Thm:main-lyapunov} and Theorem~\ref{Thm:main-lyapunov-Holder}, we need the following proposition which is a particular case of \cite[Proposition 3.1]{HLMXZ}.
 	\begin{pro}\label{prop-2} Let  $T$ be a $C^{1}$ expanding self-map of $\mathbb{T}$ with $\|DT(x)\|>1$ for all $x\in\mathbb{T}$ and $E\subset \mathbb{T}$ be a nonempty compact invariant subset of $T$. Then for any $C>0$, there exists a periodic orbit $\Gamma$ of $T$ depending on $C$ such that
 		$$G(\Gamma)>C\cdot \sum_{x\in\Gamma}d(x,E).$$
 	\end{pro}

\subsection{The perturbation result}
Given a point $x$, let $\delta_x$ be the Dirac $\delta$-measure supported on the point $x$. For a self map $S$ and a periodic orbit $\Gamma$ of $S$, denote by
$$\delta_\Gamma=\frac{1}{\#\Gamma}\sum_{x\in\Gamma}\delta_x.$$

\begin{thm}\label{Thm:perturbative}
Let $T$ be a $C^{1,1}$ expanding self-map of $\mathbb{T}$. For any $\varepsilon>0$, there are
\begin{itemize}

\item a periodic orbit $\Gamma_T$ of $T$;

\item an open set ${\mathcal U}_{\varepsilon,T}$ in the $\varepsilon$-neighborhood of $T$ for the $C^{1,1}$-topology;

\end{itemize}
such that for any $S\in{\mathcal U}_{\varepsilon,T}$, $\delta_{\Gamma_S}$ is the unique Lyapunov minimizing measure of $S$.
%
%
\end{thm}

Theorem~\ref{Thm:main-lyapunov-Holder} can be deduced from Theorem~\ref{Thm:perturbative} directly. \qed

One has the following more precise version of Theorem~\ref{Thm:perturbative}.

\begin{thm}\label{Thm:perturbative-extended}
Let $T$ be a $C^{1,1}$ expanding self-map of $\mathbb{T}$. For any $\varepsilon>0$, there are
\begin{itemize}

\item a periodic orbit $\Gamma_T$ of $T$;

\item a map $h:\TT\to\RR$ satisfying $h$ is supported in a small neighborhood of $\Gamma_T$, $\|h\|_{C^{1,1}}<\varepsilon/2$, and $S_0=T+h$ can be regarded\footnote{One can do this perturbation for the lift of $T$ from $\RR$ to $\RR$ and then pull the perturbation back to $\TT^1$.} as an $\varepsilon/2$-perturbation of $T$ for the $C^{1,1}$-topology;

\item a neighborhood ${\mathcal U}_{\varepsilon,T}$ of $S_0$ such that ${\mathcal U}_{\varepsilon,T}$ is contained in the $\varepsilon$-neighborhood of $T$;

\end{itemize}
such that for any $S\in{\mathcal U}_{\varepsilon,T}$, $\delta_{\Gamma_S}$ is the unique Lyapunov minimizing measure of $S$.%
%
\end{thm}

\medskip

Based on the above Lipschitz-$C^1$ version, one has the following differentiable $C^2$ version.
\begin{thm}\label{Thm:perturbative-smooth}
Let $T$ be a $C^{2}$ expanding self-map of $\mathbb{T}$. For any $\varepsilon>0$, there are
\begin{itemize}

\item a periodic orbit $\Gamma_T$ of $T$;

\item an open set ${\mathcal U}_{\varepsilon,T}$ in the $\varepsilon$-neighborhood of $T$ for the $C^{2}$-topology;

\end{itemize}
such that for any $S\in{\mathcal U}_{\varepsilon,T}$, $\delta_{\Gamma_S}$ is the unique Lyapunov minimizing measure of $S$.
%
%
\end{thm}
Theorem~\ref{Thm:main-lyapunov} can be deduced from Theorem~\ref{Thm:perturbative-smooth} directly. \qed

\subsection{The proof of Theorem~\ref{Thm:perturbative-extended}}
This subsection is devoted to the proof of Theorem~\ref{Thm:perturbative-extended}. So now we are under the assumptions of Theorem~\ref{Thm:perturbative-extended}.

Recall that when $X,Y$ are two metric spaces, $f:X\to Y$ is a map, the Lipschitz constant of $f$ is defined to be
$${\rm Lip}(f)=\sup_{x_1,x_2\in X, x_1\neq x_2}\frac{d_Y(f(x_1),f(x_2))}{d_X(x_1,x_2)}.$$

\begin{proof} The proof can be divided into several steps. Up to changing the metric on $\TT$, without loss of generality, one can assume that $\|DT(x)\|>1$ for any $x\in\TT$.
\paragraph{The cohomological equation.} By Lemma \ref{lem-Mane} there exists  $f\in {\rm Lip}(\mathbb{T},\RR)$ such that \eqref{AA} holds for
	$$F_T(x)=f(T(x))-f(x)+\log\|DT(x)\|.$$
	By Lemma \ref{lem-Mane}, one has that
\begin{align}\label{P-getmeasure}
F_T(x)\ge \alpha(T)~\forall x\in\mathbb{T},
\text{ and }F_T|_{{\rm supp}(\mu)}=\alpha(T),~\forall \mu\in\cL^-(T).
\end{align}

For any other self-map $S$, denote by
 $$F_S(x)=f(S(x))-f(x)+\log\|DS(x)\|.$$

\paragraph{Fix constants.} We fix a constant $K$ independent of the perturbation such that
\begin{itemize}
\item $K>\max\{2\max_{x\in\TT}\|DT(x)\|$,10\}.
\item $K>{\rm Lip}(f)\cdot (\max_{x\in\TT}\|DT(x)\|+1)+{\rm Lip}(DT)>2{\rm Lip}(f)$.
	\item $K>\lambda/(\lambda-1)$, where $\lambda=\inf_{x\in\TT}\|DT(x)\|$.

\end{itemize}


\paragraph{Reduce $\varepsilon$.} By reducing $\varepsilon$ if necessary, one has that for any $S$ satisfying $d_{C^{1,1}}(S,T)<\varepsilon$, one has that $S$ is still an expanding self-map with $\|DS(x)\|>1$ for any $x\in\TT$, and we have
\begin{enumerate}
\item \begin{align}\label{f.achoiceof-K}
\|DS(x)\|>\frac{\|DT(x)\|+1}{2}>1>2/K,~~~\forall x\in\TT.
\end{align}

\item \begin{align}\label{f.uniform-neighborhood}
	K>2\max_{x\in\TT}\|DS(x)\|,~~~K>\min_{x\in\TT}\frac{\|DS(x)\|}{\|DS(x)\|-1}.
	\end{align}
\item \begin{align}\label{f.K-control-Lipschitz}
K>{\rm Lip}(f)\cdot (\max_{x\in\TT}\|DS(x)\|+1)+{\rm Lip}(DS).
         \end{align}

\end{enumerate}


\paragraph{A small constant $\rho_\varepsilon$ and big constants $L_\varepsilon$, $C_\varepsilon$.}
		Fix $L_\varepsilon\in\mathbb{N}$ such that
\begin{align}\label{f.choose-L}
L_\varepsilon\cdot \varepsilon >4K^6.
\end{align}
Fix  $\rho_\varepsilon>0$ sufficiently small such that
\begin{align}\label{f.choose-rho}
\rho_\varepsilon \cdot K^{L_\varepsilon}<\frac{1}{2K}.
\end{align}
Fix $C_\varepsilon$ sufficiently large such that
\begin{align}\label{f.choose-C-second}
\varepsilon\cdot \rho_\varepsilon \cdot C_\varepsilon>6K^5.
\end{align}
Now one has the following estimate.
\begin{Claim}
\begin{align}\label{f.relationship-constants}
\frac{L_\varepsilon\cdot\varepsilon\cdot\rho_\varepsilon\cdot C_\varepsilon}{K^4}>3K+\rho_\varepsilon\cdot C_\varepsilon\cdot K^2.
\end{align}
\end{Claim}
\begin{proof}
By \eqref{f.choose-L}, one has that
$$\frac{L_\varepsilon\cdot\varepsilon\cdot\rho_\varepsilon\cdot C_\varepsilon}{2K^4}>\rho_\varepsilon\cdot C_\varepsilon\cdot K^2.$$
By \eqref{f.choose-C-second}, one has that
$$\frac{L_\varepsilon\cdot\varepsilon\cdot\rho_\varepsilon\cdot C_\varepsilon}{2K^4}>3 K.$$
Combining the above inequalities together, one can conclude.
\end{proof}

\paragraph{Periodic orbits with large inner distance; constants $G^*$ and $d_*$.}
Fix $\mu_T\in\cL^-(T)$. Consider $E={\rm supp}(\mu_T)$.
 By Proposition~\ref{prop-2},  there is a periodic orbit $\Gamma_T$ of $T$ such that
\begin{align}\label{f.important-estimate}
G(\Gamma_T)>C_\varepsilon \cdot\sum_{x\in\Gamma_T}d(x,{\rm supp}(\mu_T)).
\end{align}
Note that by the definition, this is still valid when $E$ is a fixed point: in this case $d(x,{\rm supp}(\mu_T))=0$, but $G(\Gamma)>0$.
By the Lipschitz property of $F_T$ and the choice of $K$, one has that for any $x\in\Gamma_T$,
\begin{align}\label{f.estimate-after-cohomology}
\left|F_T(x)-\alpha(T)\right|\le K\cdot d(x,{\rm supp}(\mu_T)).
\end{align}
Denote by $G^*=G(\Gamma_T)$ and $d_*=\sum_{x\in\Gamma_T}d(x,{\rm supp}(\mu_T))$. The inequality \eqref{f.important-estimate} can be read as $G^*>C_\varepsilon d_*$.

\paragraph{The perturbation map $h$.} Assume that  $\Gamma_T=\{p_1,p_2,\cdots,p_{\tau(p)}\}$. In a local chart, one defines the perturbations in the following way.

Since $K>\max_{x\in\TT}\|DT(x)\|$, one has that for any $x\in\TT$, $1/\|DT(x)\|>1/K$. Thus, for any $\gamma\in[0,1]$, one has that
$$\frac{1}{\|DT(x)\|-\gamma\cdot\varepsilon\cdot\rho_\varepsilon\cdot G^*/(2K)}>1/K.$$

{Consequently, we choose $\gamma_i\in[0,1]$ such that if we are in the interval $[\|DT(p_i)\|-\gamma_i\cdot\varepsilon\cdot\rho_\varepsilon \cdot G^*/K^3,\|DT(p_i)\|]$, one has that
\begin{align}\label{f.integral}
		\int_{\|DT(p_i)\|-\gamma_i\cdot\varepsilon\cdot\rho_\varepsilon \cdot G^*/(2K)}^{\|DT(p_i)\|}\frac{1}{z}{d}z=\varepsilon\cdot \rho_\varepsilon\cdot G^*/K^4.
		\end{align}}

Define a real valued function $h(x)$ on $\mathbb{T}$, such that in local charts, one has the following expression:
	\[
h(x)= \left\{
	\begin{array}{ll}
-\frac{\varepsilon}{2K \cdot (\rho_\varepsilon \cdot G^*)}(x-p_i)(p_i+\rho_\varepsilon\cdot G^*-x)^{2}\cdot\gamma_i,\ \ &\text{ if } x\in(p_i,p_{i}+\rho_\varepsilon\cdot G^*),\\
0,&\text{ if } x=p_i,\\
-\frac{\varepsilon}{2K\cdot (\rho_\varepsilon \cdot G^*)}(x-p_i)(x-p_i+\rho_\varepsilon\cdot G^*)^{2}\cdot\gamma_i,\ \ &\text{ if } x\in(p_{i}-\rho_\varepsilon\cdot G^*,p_{i}),\\	
0,& \text{ others}.
	\end{array}
	\right.
	\]


\begin{lem}\label{Lem:estimate-h}
$h$ has the following properties:
\begin{enumerate}
\item\label{i.tobezero} $h(p_i)=0$, $h(p_i\pm \rho_\varepsilon\cdot G^*)=0$ and $Dh(p_i\pm \rho_\varepsilon\cdot G^*)=0$.

\item $Dh(p_i)=-\gamma_i\cdot\varepsilon\cdot\rho_\varepsilon\cdot G^*/(2K)$.

\item $\|h\|_{C^1}<\varepsilon/2$, ${\rm Lip}(Dh)<\varepsilon/2$.
\end{enumerate}
\end{lem}
\begin{proof}
Without loss of generality, one can assume that $p_i=0$ in a local chart. One has the following calculation:
	\[
Dh(x)= \left\{
	\begin{array}{ll}
-\frac{\varepsilon}{2K \cdot (\rho_\varepsilon \cdot G^*)}\big[(\rho_\varepsilon\cdot G^*-x)^{2}-2x(\rho_\varepsilon\cdot G^*-x)\big]\cdot\gamma_i,\ \ &\text{ if } x\in(0,\rho_\varepsilon\cdot G^*),\\
-\gamma_i\cdot\varepsilon\cdot\rho_\varepsilon\cdot G^*/(2K),&\text{ if } x=0,\\
-\frac{\varepsilon}{2K\cdot (\rho_\varepsilon \cdot G^*)}\big[(x+\rho_\varepsilon\cdot G^*)^{2}+2x(\rho_\varepsilon\cdot G^*+x)\big]\cdot\gamma_i,\ \ &\text{ if } x\in(-\rho_\varepsilon\cdot G^*,0),\\	
0,& \text{ others}.
	\end{array}
	\right.
	\]
From the expression, one knows that $h(0)=0$, $h(\pm \rho_\varepsilon\cdot G^*)=0$, $Dh(\pm \rho_\varepsilon\cdot G^*)=0$ and $Dh(0)=-\gamma_i\cdot\varepsilon\cdot\rho_\varepsilon\cdot G^*/(2K)$. From the expression of $h$, one knows that $\|h\|_{C_0}<\varepsilon/2$. By a simple calculation, one has that the the whole interval,
$$\|Dh\|\le \frac{\varepsilon}{2K \cdot (\rho_\varepsilon \cdot G^*)}2(\rho_\varepsilon\cdot G^*)^2<\varepsilon/2.$$
One calculates the second derivative of $h$, which are not well-defined on $\pm\rho_\varepsilon\cdot G^*$:
	\[
D^2h(x)= \left\{
	\begin{array}{ll}
-\frac{\varepsilon}{2K \cdot (\rho_\varepsilon \cdot G^*)}\big[6x-4\rho_\varepsilon\cdot G^*\big]\cdot\gamma_i,\ \ &\text{ if } x\in(0,\rho_\varepsilon\cdot G^*),\\
\textrm{not well defined},&\text{ if } x=0,\\
-\frac{\varepsilon}{2K\cdot (\rho_\varepsilon \cdot G^*)}\big[6x+4\rho_\varepsilon\cdot G^*\big]\cdot\gamma_i,\ \ &\text{ if } x\in(-\rho_\varepsilon\cdot G^*,0),\\	
0,& \text{ others except } \pm\rho_\varepsilon\cdot G^*.
	\end{array}
	\right.
	\]

On each interval, one has that $\|D^2h\|<\varepsilon/2$ from the expression. Thus, one knows that ${\rm Lip}(Dh)<\varepsilon/2$.
\end{proof}

\paragraph{The perturbation $S_0$.}
Note that when we work in local charts, we can write
	$$S_0(x)=T(x)+h(x).$$
It is clear that when $\varepsilon$ is small enough, $S_0$ is an expanding self-map.
\begin{lem}$S_0$ has the following properties:
\begin{itemize}
\item $d_{C^{1,1}}(S_0,T)<\varepsilon/2$.
\item $DS_0(p_i)=DT(p_i)-\gamma_i\cdot\varepsilon\cdot\rho_\varepsilon\cdot G^*/(2K)$ for each $p_i\in\Gamma_T$.
\item $\Gamma_T$ is still the periodic orbit of $S_0$, and $T|_{\Gamma_T}=S_0|_{\Gamma_T}$.
\end{itemize}
\end{lem}
\begin{proof}
These properties follows from the properties of $h$ directly.
\end{proof}

For $S_0$, one has that for any $p\in\Gamma_T$, which is also a periodic point of $S_0$,
\begin{equation}\label{f.first-perturbation-periodic}\begin{split}
 F_{S_0}(p)&=f(T(p))-f(p) +\log\|DS_0(p)\|\\
&=f(T(p))-f(p) +\log\|DT(p)\|-\int^{\|DT(p)\|}_{\|DS_0(p)\|}\frac{1}{z}{d}z\\
&\overset{\eqref{f.integral}}=f(T(p))-f(p) +\log\|DT(p)\|-\int^{\|DT(p)\|}_{\|DT(p)\|-\gamma_i\cdot\varepsilon\cdot\rho_\varepsilon\cdot G^*/(2K)}\frac{1}{z}{d}z\\
&=F_T(x)-\varepsilon\cdot\rho_\varepsilon\cdot G^*/K^4.
\end{split}	
\end{equation}
Thus, for any $p,q\in\Gamma_{S_0}$, one has that
\begin{align}\label{f.first-perturb-periodic-minus}
\begin{split}
|F_{S_0}(p)-F_{S_0}(q)|\overset{\eqref{f.first-perturbation-periodic}}=&|F_{T}(p)-F_{T}(q)|\\
\overset{\eqref{f.estimate-after-cohomology}}\le& K(d(p,{\rm supp}(\mu_T))+d(q,{\rm supp}(\mu_T))).
\end{split}
\end{align}

\paragraph{Choose the constants $\widetilde\varepsilon_0>\widetilde\varepsilon>0$ and find the neighborhood $\mathcal U$.}
Take $\widetilde\varepsilon_0\in(0,\varepsilon/2)$ such that
\begin{align}\label{f.choice-epsilon-0}
(\rho_\varepsilon\cdot G^*+\widetilde\varepsilon_0)\cdot K^{L_\varepsilon}<\frac{G^*-2\widetilde\varepsilon_0}{2K}.
\end{align}
and
\begin{align}\label{f.second-choice-epsilon-0}
\begin{split}
L_\varepsilon\left(\varepsilon\cdot \rho_\varepsilon\cdot G^*/K^4-K\cdot d_*-2\widetilde\varepsilon_0\cdot K\right)&>\big(\rho_\varepsilon G^*+ \widetilde \varepsilon_0\big)K^2\\
& +\tau(\Gamma_S)\left(4\widetilde\varepsilon_0\cdot K+2Kd_*/\tau(\Gamma_S)\right).
\end{split}
\end{align}
\begin{align}\label{f.third-condition-epsilon-0}
\varepsilon\cdot \rho_\varepsilon\cdot G^*/K^4-K\cdot d_*-2\widetilde\varepsilon_0\cdot K>0.
\end{align}

\begin{Claim}
One can choose $\widetilde\varepsilon_0$ such that Inequalities \eqref{f.choice-epsilon-0}, \eqref{f.second-choice-epsilon-0} and \eqref{f.third-condition-epsilon-0} hold.
\end{Claim}
\begin{proof}[Proof of the Claim]
These come from \eqref{f.choose-rho}, \eqref{f.relationship-constants} and \eqref{f.choose-C-second} by noticing $G^*>C_\varepsilon d_*$.

\end{proof}

By Theorem~\ref{Thm:conjugacy}, there is $\widetilde\varepsilon>0$ such that for any $S$, if $d_{C^1}(S,S_0)<\widetilde\varepsilon$, then there is a homeomorphism $\pi_S:\TT\to\TT$ such that
\begin{align}\label{property-homeomorphism}
 d_{C^0}(\pi_S,Id)<\widetilde\varepsilon_0~ \textrm{~and~} ~\pi_S\circ S_0=S\circ \pi_S.
\end{align}
 Consequently, $\Gamma_S=\pi_S(\Gamma_{S_0})$ is a periodic orbit of $S$.

 \medskip

 Without loss of generality, one can assume that $\widetilde\varepsilon<\widetilde\varepsilon_0$.

 \medskip

One has the following estimate on $\Gamma_S$:
\begin{lem}\label{Lem:inner-distance-S}
For any two distinct $x,y\in \Gamma_S$, one has that $d(x,y)>G^*-2\widetilde\varepsilon_0$.

\end{lem}
\begin{proof}
By the definition of $G^*$, one has that $d(\pi_S^{-1}(x),\pi_S^{-1}(y))>G^*$. One can conclude by noticing that $d(x,\pi_S^{-1}(x))<\widetilde\varepsilon_0$ and $d(y,\pi_S^{-1}(y))<\widetilde\varepsilon_0$.
\end{proof}

%
%
%
We take $\mathcal U$ to be the $\widetilde\varepsilon$-neighborhood of $S_0$ in the $C^{1,1}$-topology.

\begin{Claim}

$\mathcal U$ is contained in the $\varepsilon$-neighborhood of $T$ in the $C^{1,1}$-topology.
\end{Claim}
\begin{proof}[Proof of the Claim]
This follows from the fact that $0<\widetilde\varepsilon<\widetilde\varepsilon_0<\varepsilon/2$.
\end{proof}

For any $S,R\in\mathcal U$, for any $x\in \TT$, by \eqref{f.achoiceof-K}, one has that
\begin{align*}&
|\log\|DS(x)\|-\log\|DR(x)\||\\
\le& \max\{\frac{1}{\inf_{w\in\TT}\|DS(w)\|},\frac{1}{\inf_{w\in\TT}\|DR(w)\|}\}\cdot d_{C^0}(DS(x),DR(x))\\
\le&K/2\cdot d_{C^0}(DS(x),DR(x)).
\end{align*}
Hence, together with  the fact that $K>2{\rm Lip}(f)$,
\begin{align}\label{f.two-minus}
\begin{split}
&~~~~|F_S(x)-F_R(x)|\\
&\le |f(S(x))-f(R(x))|+|\log\|DS(x)\|-\log\|DR(x)\||\\
&\le {\rm Lip}(f)\cdot d_{C^0}(S,R)+K/2\cdot d_{C^0}(DS(x),DR(x))\\
&\le \widetilde\varepsilon_0\cdot K.
\end{split}
\end{align}


\paragraph{The average on $\Gamma_S$.}For any $S\in{\mathcal U}$, denote by

$$A_{\Gamma_{S}}=\int F_S{d}\delta_{\Gamma_S}=\frac{\sum_{z\in\Gamma_{S}}F_S(z)}{\#\Gamma_{S}}.$$

Clearly, for $S_0$, by \eqref{f.first-perturbation-periodic}, one has that
\begin{align}\label{f.difference-periodic}
\begin{split}
A_{\Gamma_{S_0}}&=\frac{\sum_{z\in\Gamma_{T}}F_T(z)}{\#\Gamma_{S}}-\varepsilon\cdot\rho_\varepsilon\cdot G^*/K^4\\
&\overset{\eqref{f.estimate-after-cohomology}}\le \alpha(T)+K\cdot \frac{d_*}{\#\Gamma_T}-\varepsilon\cdot\rho_\varepsilon\cdot G^*/K^4.
\end{split}
\end{align}
Thus, for any $S\in \mathcal U$, one has that
\begin{align}\label{f.estimate-on-A-Gamma_S}
A_{\Gamma_S}\overset{\eqref{f.two-minus}}\le A_{\Gamma_{S_0}}+K\cdot\widetilde\varepsilon_0\overset{\eqref{f.difference-periodic}}\le \alpha(T)+K\cdot\widetilde\varepsilon_0+K\cdot \frac{d_*}{\#\Gamma_T}-\varepsilon\cdot\rho_\varepsilon\cdot G^*/K^4.
\end{align}

\paragraph{Estimates for $\Gamma_S$.}
For any $x,y\in \Gamma_S$, one has that
{\small\begin{align}\label{f.newsystem-difference-periodic}
\begin{split}
\left|F_S(x)-F_S(y)\right|
\le&\left|F_S(x)-F_{S_0}(x)\right|+\left|F_S(y)-F_{S_0}(y)\right|+\left|F_{S_0}(x)-F_{S_0}(y)\right|\\
\overset{\eqref{f.two-minus}}\le& 2\widetilde\varepsilon_0\cdot K+\left|F_{S_0}(x)-F_{S_0}(\pi_S^{-1}(x))\right|+\left|F_{S_0}(y)-F_{S_0}(\pi_S^{-1}(y))\right|\\
&~~~~~~~+\left|F_{S_0}(\pi_S^{-1}(x))-F_{S_0}(\pi_S^{-1}(y))\right|\\
\overset{\eqref{f.first-perturb-periodic-minus}}\le& 2\widetilde\varepsilon_0\cdot K+2\widetilde\varepsilon_0\cdot K+K(d(\pi_S^{-1}(x),{\rm supp}(\mu_T))+d(\pi_S^{-1}(y),{\rm supp}(\mu_T)))\\
=&4\widetilde\varepsilon_0\cdot K+K(d(\pi_S^{-1}(x),{\rm supp}(\mu_T))+d(\pi_S^{-1}(y),{\rm supp}(\mu_T))).
\end{split}
\end{align}}

%

Thus, for each $x\in\Gamma_S$, one has that
{\small\begin{align}\label{f.periodic-sum-average}
\begin{split}
\left|F_S(x)- A_{\Gamma_S}\right|
\le&\frac{1}{\tau(\Gamma_S)}\sum_{y\in\Gamma_S}\left|F_S(x)-F_S(y)\right|\\
\overset{\eqref{f.newsystem-difference-periodic}}\le& \frac{1}{\tau(\Gamma_S)}\sum_{y\in\Gamma_S}\big(4\widetilde\varepsilon_0\cdot K+K(d(\pi_S^{-1}(x),{\rm supp}(\mu_T))+d(\pi_S^{-1}(y),{\rm supp}(\mu_T)))\big)\\
=&4\widetilde\varepsilon_0\cdot K+K(d(\pi_S^{-1}(x),{\rm supp}(\mu_T))+\frac{1}{\tau(\Gamma_S)}\sum_{y\in\Gamma_S}d(\pi_S^{-1}(y),{\rm supp}(\mu_T)))\\
=&4\widetilde\varepsilon_0\cdot K+K(d(\pi_S^{-1}(x),{\rm supp}(\mu_T))+d_*/\tau(\Gamma_S).
\end{split}
\end{align}}
Define  ${\widetilde F}_S(x)=F_S(x)-A_{\Gamma_{S}}$ for all $x\in \mathbb{T}$.
Thus,
\begin{align}\label{f.integral-G}
\frac{\sum_{z\in\Gamma_{S}}{\widetilde F}_S(z)}{\#\Gamma_{S}}=\int {\widetilde F}_Sd\delta_{\Gamma_{S}}{=}0.
\end{align}
Moreover, recall that $F_S=f\circ S-f+\log\|DS\|$, one has
\begin{align}\label{f.cohomology-bounded-by-K}
\begin{split}
{\rm Lip}({\widetilde F}_S)&={\rm Lip}(F_S)\le {\rm Lip}(f)\max_{x\in\TT}\|DS(x)\|+{\rm Lip}(f)+\frac{1}{\min_{x\in\TT}\|DS(x)\|}{\rm Lip}(DS)\\
&\le {\rm Lip}(f)(\max_{x\in\TT}\|DS(x)\|+1)+{\rm Lip}(DS)\overset{\eqref{f.K-control-Lipschitz}}<K
\end{split}
\end{align}

\paragraph{Domains away from the periodic orbit.}Put
\begin{align*}\mathcal{F}_{T}=\{x\in\mathbb{T}:d(x,\Gamma_T)> \rho_\varepsilon\cdot G^*\}.\end{align*}
Then $\mathcal{F}_{T}$ is an open subset of $\mathbb{T}$. By the definition of $h(x)$, one can see that $h(x)=0$ for any $x\in \mathcal{F}_{T}$ and hence
\begin{align}\label{45}F_{S_0}|_{\mathcal{F}_{T}}=F_T|_{\mathcal{F}_{T}}.\end{align}

\paragraph{Estimates in ${\mathcal F}_T$.}We give a lower bound of ${\widetilde F}_S$ in ${\mathcal F}_T$.
\begin{Claim}
\begin{align}\label{f.first-estimite-G}
{\widetilde F}_S(x)\ge \varepsilon\cdot \rho_\varepsilon\cdot G^*/K^4-K\cdot d_*-2\widetilde\varepsilon_0\cdot K>0, ~\forall x\in \mathcal{F}_T.
\end{align}
\end{Claim}
\begin{proof}[Proof of the Claim]
Since $F_{S_0}(x)\overset{\eqref{45}}=F_T(x)\ge \alpha(T)$ for any $x\in {\mathcal F}_T$, one has that
\begin{align}\label{f.estimate-FS-FT}
F_S(x)\overset{\eqref{f.two-minus}}\ge F_{S_0}(x)-\widetilde\varepsilon_0\cdot K\ge \alpha(T)-\widetilde\varepsilon_0\cdot K,~~~\forall x\in{\mathcal F}_T.
\end{align}
This implies that
\begin{align*}
&~~~~~~~~{\widetilde F}_S(x)=F_S(x)-A_{\Gamma_S}\\
&\overset{\eqref{f.estimate-FS-FT},\eqref{f.estimate-on-A-Gamma_S}}\ge (\alpha(T)-\widetilde\varepsilon_0\cdot K)-(\alpha(T)+K\cdot\widetilde\varepsilon_0+K\cdot \frac{d_*}{\#\Gamma_T}-\varepsilon\cdot\rho_\varepsilon\cdot G^*/K^4)\\
&~~~\ge \varepsilon\cdot\rho_\varepsilon\cdot G^*/K^4-K\cdot d_*-2\widetilde\varepsilon_0\cdot K\overset{\eqref{f.third-condition-epsilon-0}}>0.
\end{align*}
\end{proof}

\paragraph{Estimates outside of ${\mathcal F}_T$}
\begin{lem}\label{Lem:sum-see-positivity}
If $x\notin\mathcal{F}_T\cup\Gamma_S$, then there exists $N(x)\in\mathbb{N}$ such that
$$x,S(x),\cdots,S^{N(x)-1}x\in\{z\in\mathbb{T}:d(z,\Gamma_{S})\le (\rho_\varepsilon\cdot G^*+\widetilde \varepsilon_0)K^{L_\varepsilon}\}$$ and $\sum_{i=0}^{N(x)-1}{\widetilde F}_S(S^i(x))>0$ .

\end{lem}
\begin{proof}
Now we assume that $x\notin \mathcal{F}_T\cup\Gamma_S$. This implies that $d(x,\Gamma_T)\le \rho_\varepsilon\cdot G^*$. By the property of $\pi_S$ \eqref{property-homeomorphism} and the fact $\Gamma_T=\Gamma_{S_0}$, one has that
\begin{align*}
d(x,\Gamma_{S})\le d(x,\Gamma_T)+ \widetilde\varepsilon_0
\le \rho_\varepsilon\cdot G^*+ \widetilde\varepsilon_0.
\end{align*}
Notice that
\begin{align}\label{f.7}\rho_\varepsilon\cdot G^*+ \widetilde \varepsilon_0<(\rho_\varepsilon\cdot G^*+ \widetilde \varepsilon_0)K^{L_\varepsilon}\overset{\eqref{f.choice-epsilon-0}}<\frac{G^*-2\widetilde \varepsilon_0}{2K}.\end{align}
By Lemma~\ref{shadow}, Lemma~\ref{Lem:inner-distance-S} and the assumption $x\notin \Gamma_{S}$, there exists $n\in\mathbb{N}\cup\{0\}$ such that
\begin{align*}d(S^n(x),\Gamma_{S})\ge \frac{G(\Gamma_{S})}{2\max_{x\in\TT}\|DS(x)\|}\overset{\eqref{f.uniform-neighborhood}}\ge \frac{G^*-2\widetilde\varepsilon_0}{2K}\overset{\eqref{f.7}}>\rho_\varepsilon\cdot G^*+ \widetilde \varepsilon_0.\end{align*}

Notice that $S$ is uniformly expanding and the expanding rate of $S$ is not larger than $K$. Thus, there exists $m(x)\in\mathbb{N}$, the minimal natural number such that
$$\rho_\varepsilon\cdot G^*+ \widetilde\varepsilon_0<d(S^{m(x)}x,\Gamma_{S})\le (\rho_\varepsilon\cdot G^*+ \widetilde\varepsilon_0)\cdot K,$$
i.e.,
\begin{align*}
m(x)=\min\{m\in\NN:~d(S^{m}x,\Gamma_{S})> \rho_\varepsilon\cdot G^*+ \widetilde \varepsilon_0\}.
\end{align*}
By the minimality of $m(x)$, one has that
\begin{align}\label{f.first-segment}
d(S^{l}x,\Gamma_{S})\le \rho_\varepsilon\cdot G^*+ \widetilde \varepsilon_0,~~~\forall ~0\le l\le m(x)-1.
\end{align}

Since the expanding rate of $S$ is bounded by $K$, one has
\begin{align*}
 S^{m(x)}(x),S^{m(x)+1}(x),\cdots,&S^{m(x)+L_\varepsilon-1}(x)\\
 &\in \{z\in\mathbb{T}:\rho_\varepsilon\cdot G^*+\widetilde \varepsilon_0<d(z,\Gamma_{S})\le (\rho_\varepsilon\cdot G^*+\widetilde \varepsilon_0)K^{L_\varepsilon}\}.
\end{align*}
Choose $L(x)\in\NN$ to be the minimal integer such that $d(S^{m(x)+L(x)}(x),\Gamma_{S})>(\rho_\varepsilon\cdot G^*+\widetilde \varepsilon_0)K^{L_\varepsilon}$. Clearly by the definition, one has that $L(x)\ge L_\varepsilon$. Now we set $N(x)=m(x)+L(x)$.

\medskip

Now we are going to prove  $\sum_{i=0}^{m(x)+L(x)-1}{\widetilde F}_S(S^i(x))>0$.
Let $m\in\{0,1,\cdots,m(x)-1\}$ be the maximal number such that
$$S^{m}(x)\notin \mathcal{F}_T.$$
Then by \eqref{f.first-estimite-G}, one has that
\begin{align}\label{X-12}
{\widetilde F}_S(S^i(x))>\frac{\varepsilon\cdot \rho_\varepsilon\cdot G^*}{K^4}-2\widetilde\varepsilon_0\cdot K-K\cdot d_*>0,~~~\forall~m\le i\le m(x)+L(x)-1.\end{align}
Choose $p_{S}\in\Gamma_{S}$ such that
$$d(x,\Gamma_{S})=d(x,p_{S}).$$
\begin{Claim}
For $0\le i\le m-1$, one has that $$d(S^i(x),\Gamma_{S})=d(S^i(x),S^i(p_{S})).$$
\end{Claim}
\begin{proof}[Proof of the Claim]
Given $0\le i\le m$, one knows that by Lemma~\ref{Lem:inner-distance-S},
$$d(S^i(p_S),\Gamma_S\setminus\{S^i(p_S)\})>G^*-2\widetilde\varepsilon_0.$$ Thus, it suffices to prove that $d(S^i(x),S^{i}(p_{S}))<G^*/2-\widetilde\varepsilon_0$. By the fact that $i\le m-1< m(x)$, one has that
$$d(S^i(x),S^{i}(p_{S}))\le  \rho_\varepsilon\cdot G^*+\widetilde \varepsilon_0\le(\rho_\varepsilon\cdot G^*+\widetilde \varepsilon_0)K^{L_\varepsilon}\overset{\eqref{f.choice-epsilon-0}}< \frac{G^*-2\widetilde\varepsilon_0}{2K}<G^*/2-\widetilde\varepsilon_0.$$
\end{proof}
%
%

By the above claim, we have
\begin{align*}
\sum_{i=0}^{m-1}d(S^i(x),S^i(p_S))\le& \sum_{i=0}^{m-1} \frac{d(S^{m-1}(x),S^{m-1}(p_S))}{(\min_{w\in\TT}\|DS(w)\|)^i}\\
\le&\frac{(\rho_\varepsilon\cdot G^*+ \widetilde \varepsilon_0)\cdot\min_{w\in\TT}\|DS(w)\|}{\min_{w\in\TT}\|DS(w)\|-1}\overset{\eqref{f.uniform-neighborhood}}<(\rho_\varepsilon\cdot G^*+ \widetilde \varepsilon_0)K.
\end{align*}
Therefore,
\begin{align}\label{f.close-to-periodic}
\begin{split}\sum_{i=0}^{m-1}\big({\widetilde F}_S(S^i(x))-{\widetilde F}_S(S^i(p_S))\big)\overset{\eqref{f.cohomology-bounded-by-K}}\ge -K\sum_{i=0}^{m-1}d(S^i(x),S^i(p_S))
\ge -K^2(\rho_\varepsilon\cdot G^*+\widetilde\varepsilon_0).
\end{split}
\end{align}

\begin{Claim}
One has that
\begin{align}\label{f.periodic-partial-sum}
\sum_{i=0}^{m-1}{\widetilde F}_S(S^i(p_S))\ge -\tau(\Gamma_S)\left(4\widetilde\varepsilon_0\cdot K+2Kd_*/\tau(\Gamma_S)\right).
\end{align}
\end{Claim}
\begin{proof}[Proof of the claim]
Assume that $m=Q\tau(\Gamma_S)+r$ for some nonnegative integer $Q$ and $0\le r\le \tau(\Gamma_S)-1$. When $r=0$, one has that
\begin{align*}
\begin{split}
\sum_{i=0}^{m-1}{\widetilde F}_S(S^i(p_S))&=Q\cdot \sum_{i=0}^{\tau(\Gamma_S)-1}{\widetilde F}_S(S^i(p_S))=0\\
&\ge-\tau(\Gamma_S)\left(4\widetilde\varepsilon_0\cdot K+2Kd_*/\tau(\Gamma_S)\right).
\end{split}
\end{align*}
When $r\ge1$,
\begin{align*}
\begin{split}
\sum_{i=0}^{m-1}{\widetilde F}_S(S^i(p_S))&=Q\cdot \sum_{i=0}^{\tau(\Gamma_S)-1}{\widetilde F}_S(S^i(p_S))+\sum_{i=Q\tau(\Gamma_S)}^{Q\tau(\Gamma_S)+r-1}{\widetilde F}_S(S^i(p_S))\\
&\overset{\eqref{f.integral-G}}=\sum_{i=0}^{r-1}{\widetilde F}_S(S^i(p_S))\\
&=\sum_{i=0}^{r-1}F_S(S^i(p_S))-rA_{\Gamma_{S}}\\
&\ge -\sum_{z\in\Gamma_{S}} \left|F_S(z)-A_{\Gamma_{S}}\right|\\
&\overset{\eqref{f.periodic-sum-average}}\ge-\sum_{z\in\Gamma_{S}}\left(4\widetilde\varepsilon_0\cdot K+K(d(\pi_S^{-1}(z),{\rm supp}(\mu_T))+d_*/\tau(\Gamma_S)\right)\\
&\ge-\tau(\Gamma_S)\left(4\widetilde\varepsilon_0\cdot K+2Kd_*/\tau(\Gamma_S)\right).
\end{split}
\end{align*}
\end{proof}

Therefore, by applying \eqref{X-12}, \eqref{f.close-to-periodic} and \eqref{f.periodic-partial-sum},  we have that
\begin{align*}
&\hskip0.5cm \sum_{i=0}^{N(x)-1}{\widetilde F}_S(S^i(x))\\
&\ge\sum_{i=m}^{m(x)+L(x)-1}{\widetilde F}_S(S^i(x))+\sum_{i=0}^{m-1}\big({\widetilde F}_S(S^i(x))-{\widetilde F}_S(S^i(p_S))\big)+\sum_{i=0}^{m-1}{\widetilde F}_S(S^i(p_S)))\\
&\ge  (L(x)+m(x)-m)\left(\varepsilon\cdot \rho_\varepsilon\cdot G^*/K^4-K\cdot d_*-2\widetilde\varepsilon_0\cdot K\right)\\
&\ \ \ \ \ - \big(\rho_\varepsilon G^*+ \widetilde \varepsilon_0\big)K^2 -\tau(\Gamma_S)\left(4\widetilde\varepsilon_0\cdot K+2Kd_*/\tau(\Gamma_S)\right)\\
&\ge  L_\varepsilon\left(\varepsilon\cdot \rho_\varepsilon\cdot G^*/K^4-K\cdot d_*-2\widetilde\varepsilon_0\cdot K\right)\\
&\ \ \ \ \ - \big(\rho_\varepsilon G^*+ \widetilde \varepsilon_0\big)K^2 -\tau(\Gamma_S)\left(4\widetilde\varepsilon_0\cdot K+2Kd_*/\tau(\Gamma_S)\right)\\
&\overset{\eqref{f.second-choice-epsilon-0}}>0.
\end{align*}
\end{proof}

\paragraph{Ergodic measures.}Now we check for measures. Take an $S$-ergodic probability measure $\mu\neq \delta_{\Gamma_{S}}$. To conclude, it suffices to prove that
$$\int F_S d\mu>\int F_S d\delta_{\Gamma_{S}},$$
which is equivalent to to show that
\begin{align*}
\int {\widetilde F}_Sd\mu>\int {\widetilde F}_Sd\delta_{\Gamma_{S}}\overset{\eqref{f.integral-G}}=0.
\end{align*}

 Let $x$ be a generic point of $\mu$. Then $S^i(x)\not \in \Gamma_{S}$  for all $i\in\mathbb{N}\cup\{0\}$ since $\mu\neq \delta_{\Gamma_{S}}$.
 By Lemma~\ref{Lem:sum-see-positivity}, for $x\notin\cF_T$, one can define $N(x)$. So for any $y\notin\Gamma_S$, define
  \[ I(y) = \left\{
 \begin{array}{rl}1~~~~~~, & \text{if $y\in{\mathcal F}_T$} ,\\
~~~N(y)~~~, & \text{if $y\notin{\mathcal F}_T$}.
 \end{array} \right. \]

 \begin{Claim}\label{cor:index-positivy}
\begin{align}\label{f.index-positivy}
\sum_{i=0}^{I(y)-1}{\widetilde F}_S(S^i(y))>0,~~~\forall y\notin\Gamma_S.
\end{align}
 \end{Claim}
\begin{proof}[Proof of the Claim]
By Lemma~\ref{Lem:sum-see-positivity}, one has that $\sum_{i=0}^{I(y)-1}{\widetilde F}_S(S^i(y))>0$ for any $y\notin{\mathcal F}_T\cup\Gamma_S$. If $y\in{\mathcal F}_T$, $I(y)=1$. One only has to show that ${\widetilde F}_S(y)>0$ for $y\in{\mathcal F}_T$. This is given by Inequality~(\ref{f.first-estimite-G}).
\end{proof}
 Now we define an index sequence $\{j_n\}_{n\in\mathbb{N}}$ by induction on $n$. Put
 $$j_1=0,~\textrm{and}~j_n=j_{n-1}+I(S^{j_{n-1}}(x))~\textrm{for}~n\ge 2.$$
The index sequence $\{j_n\}_{n\in\mathbb{N}}$ is well defined since  $S^i(x)\not \in \Gamma_{S}$  for all $i\in\mathbb{N}\cup\{0\}$.

\medskip


Set ${\mathcal F}_1=\{z\in\mathbb{T}:d(z,\Gamma_{S})>(\rho_\varepsilon\cdot G^*+\widetilde \varepsilon_0)K^{L_\varepsilon}\}$. By the definition, one can check that ${\mathcal F}_1\subset {\mathcal F}_T$.
\begin{Claim}
\begin{align}\label{f.measure-F_1}\mu( \mathcal{F}_1)>0.\end{align}
\end{Claim}
\begin{proof}
By Lemma~\ref{shadow}, there is $n\in\NN$ such that
{\small $$d(S^n(x),\Gamma_S)\ge\frac{G(\Gamma_S)}{2\max_{z\in\TT}\|DS(z)\|}\overset{\eqref{f.uniform-neighborhood}}>\frac{G(\Gamma_S)}{2K}\overset{\textrm{Lemma}~\ref{Lem:inner-distance-S}}>\frac{G^*-2\widetilde\varepsilon_0}{2K}\overset{\eqref{f.choice-epsilon-0}}>(\rho_\varepsilon\cdot G^*+\widetilde \varepsilon_0)K^{L_\varepsilon}.$$}
In other words, $S^n(x)\in\cF_1$. Since $S^n(x)$ is also  a generic point of $\mu$, one has that $\mu(\cF_1)>0$.
\end{proof}


Put
$$\mathcal{N}=\{i\in\mathbb{N}\cup\{0\}:S^i(x)\in\mathcal{F}_1\}.$$
By the ergodicity of $\mu$, we have that
\begin{align}\label{f.ertodicity-density}\liminf_{N\to+\infty}\frac{\#(\mathcal{N}\cap[0,N-1])}{N}\ge \mu(\mathcal{F}_1).\end{align}
By the fact that ${\mathcal F}_1\subset{\mathcal F}_T$ and  the definition of $I$, one has that
\begin{align}\label{f.property-N}
\mathcal{N}\subset\{j_n:j_{n+1}-j_{n}=1,n\in\mathbb{N}\}.
\end{align}
Therefore, we have
\begin{align*}
\begin{split}
\int {\widetilde F}_Sd\mu=&\lim_{m\to+\infty}\frac{1}{j_{m+1}}\sum_{i=0}^{j_{m+1}-1}{\widetilde F}_S(S^i(x))\\
=&\lim_{m\to+\infty}\frac{1}{j_{m+1}}\sum_{n=1}^m \sum_{i=j_n}^{j_{n+1}-1}{\widetilde F}_S(S^i(x))\\
=&\lim_{m\to+\infty}\frac{1}{j_{m+1}}(\sum_{n=1}^m \sum_{S^{j_n}(x)\in {\mathcal F}_1,i=j_n}^{j_{n+1}-1}{\widetilde F}_S(S^i(x))+\sum_{n=1}^m \sum^{j_{n+1}-1}_{S^{j_n}(x)\notin {\mathcal F}_1,i=j_n}{\widetilde F}_S(S^i(x)))\\
\overset{{\eqref{f.index-positivy}}}\ge&\liminf_{m\to+\infty}\frac{1}{j_{m+1}}(\sum_{n=1}^m \sum_{S^{j_n}(x)\in {\mathcal F}_1,i=j_n}^{j_{n+1}-1}{\widetilde F}_S(S^i(x))\\
\overset{(\ref{f.property-N})}=& \liminf_{m\to+\infty} \frac{1}{j_{m+1}}\sum \limits_{n\in [1,m]:S^{j_n}(x)\in{\mathcal F}_1} {\widetilde F}_S(S^{j_n}(x))\\
\overset{(\ref{f.first-estimite-G})}\ge&  \liminf_{m\to+\infty} \frac{1}{j_{m+1}}\sum \limits_{n\in [1,m]:S^{j_n}(x)\in{\mathcal F}_1} \left(\frac{\varepsilon\cdot \rho_\varepsilon\cdot G^*}{K^4}-2\widetilde\varepsilon_0\cdot K-K\cdot d_*\right)\\
 \ge& \liminf_{m\to+\infty} \frac{|[0,j_{m+1}-1]\cap \mathcal{N}|}{j_{m+1}} \left(\frac{\varepsilon\cdot \rho_\varepsilon\cdot G^*}{K^4}-2\widetilde\varepsilon_0\cdot K-K\cdot d_*\right)\\
\overset{\eqref{f.ertodicity-density}}\ge &\mu(\mathcal{F}_1)\left(\frac{\varepsilon\cdot \rho_\varepsilon\cdot G^*}{K^4}-2\widetilde\varepsilon_0\cdot K-K\cdot d_*\right)
\overset{\eqref{f.measure-F_1}\eqref{f.first-estimite-G}}>0.
\end{split}
\end{align*}
Hence one can conclude.
\end{proof}

\subsection{Proof of Theorem~\ref{Thm:perturbative-smooth}} Let $T$ be a $C^2$ expanding self-map of $\TT$. Consider the $C^{1,1}$ map $h$ defined in the proof of Theorem~\ref{Thm:perturbative-extended}. For $\delta>0$ we let
 	$$h_\delta(x)=\frac{1}{2\delta}\int_{-\delta}^{\delta}h(x+s)ds.$$
One has that
$$Dh_\delta(x)=\frac{1}{2\delta}\int_{-\delta}^{\delta}Dh(x+s)ds.$$
By a simple calculation, one has that
\begin{Claim}
For any $x\in\TT$, $D^2h_\delta(x)=1/(2\delta)(Dh(x+\delta)-Dh(x-\delta))$.
\end{Claim}
By  Theorem~\ref{Thm:perturbative-extended}, $\|D^2h_\delta\|_{C^0}\le {\rm Lip}(Dh)<\varepsilon/2$. Clearly, $\|h_\delta\|_{C^1}<\varepsilon/2$. Thus $S_\delta=T+h_\delta$ is an $\varepsilon/2$-perturbation of $T$ for $\delta$ small enough in the $C^2$ topology.

Any $C^2$-small perturbation of $S_\delta$ is contained in a neighborhood of $S_0$ in the $C^{1,1}$-neighborhood. Thus by Theorem~\ref{Thm:perturbative-extended}, for any $S$ sufficiently close to $S_\delta$, the Lyapunov minimizing measure of $S$ is supported on $\Gamma_S$. Hence the proof of Theorem~\ref{Thm:perturbative-smooth} is complete.\qed

\noindent Wen Huang\\

\noindent CAS Wu Wen-Tsun Key Laboratory of Mathematics 

\noindent School of Mathematical Sciences

\noindent University of Science and Technology of China, Hefei, Anhui,
230026, PR China

\noindent wenh@mail.ustc.edu.cn

\bigskip

\noindent Leiye Xu\\

\noindent CAS Wu Wen-Tsun Key Laboratory of Mathematics

\noindent School of Mathematical Sciences

\noindent University of Science and Technology of China, Hefei, Anhui,
230026, PR China

\noindent leoasa@mail.ustc.edu.cn

\bigskip

\noindent Dawei Yang

\noindent School of Mathematical Sciences

\noindent Soochow University, Suzhou, 215006, P.R. China

\noindent yangdw1981@gmail.com, yangdw@suda.edu.cn

%
%
%
%

	\end{document}